\documentclass[11pt]{article}
\usepackage{amsmath}
\usepackage{mathrsfs}
\usepackage{amssymb}
\usepackage{amsfonts}
\usepackage{bm}
\usepackage{mathtools}
\usepackage{tikz}
\usetikzlibrary{calc}

\usepackage{xcolor}
\usepackage{geometry}
\usepackage[pdfstartview=FitH,
bookmarksopen=false,
colorlinks=true,
linkcolor=blue,
citecolor=blue,
urlcolor=blue]{hyperref}

\makeatletter
\def\@seccntDot{.}
\def\@seccntformat#1{\csname the#1\endcsname\@seccntDot\hskip 0.5em}
\renewcommand\section{\@startsection{section}{1}{\z@}%
{18\p@ \@plus 6\p@ \@minus 3\p@}%
{9\p@ \@plus 6\p@ \@minus 3\p@}%
{\large\bfseries\boldmath}}
\renewcommand\subsection{\@startsection{subsection}{2}{\z@}%
{12\p@ \@plus 6\p@ \@minus 3\p@}%
{3\p@ \@plus 6\p@ \@minus 3\p@}%
{\bfseries\boldmath}}
\renewcommand\subsubsection{\@startsection{subsubsection}{3}{\z@}%
{12\p@ \@plus 6\p@ \@minus 3\p@}%
{\p@}%
{\bfseries\boldmath}}
\makeatother

\usepackage{microtype}

\usepackage[amsmath,thmmarks]{ntheorem}
\theoremstyle{plain}
\newtheorem{theorem}{Theorem}[section]
\newtheorem{lemma}{Lemma}[section]
\newtheorem{corollary}{Corollary}[section]
\newtheorem{proposition}{Proposition}[section]

\theorembodyfont{\normalfont}
\newtheorem{definition}{Definition}[section]
\newtheorem{remark}{Remark}[section]
\newtheorem{example}{Example}[section]
\newtheorem{problem}{Problem}[section]

\theoremstyle{nonumberplain}
\theoremseparator{}
\theoremsymbol{\ensuremath{\Box}}
\newtheorem{proof}{\it Proof.}

\numberwithin{equation}{section}
\allowdisplaybreaks
\begin{document}

\title{The $\alpha$-normal labeling method for computing the $p$-spectral radii of uniform hypergraphs}
\author{Lele Liu 
\thanks{Department of Mathematics, Shanghai University, Shanghai 200444, P.R. China
({\tt ahhylau@gmail.com}). The work is done when this author visited the University 
of South Carolina during September 2017\,--\,March 2019 under the support of the fund 
from the China Scholarship Council (CSC No. 201706890045).}
\and Linyuan Lu
\thanks{Department of Mathematics, University of South Carolina, Columbia, SC 29208, 
USA ({\tt lu@math.sc.edu}). This author was supported in part by NSF grant DMS-1600811
and ONR grant N00014-17-1-2842.}}

\maketitle

\begin{abstract}
Let $G$ be an $r$-uniform hypergraph of order $n$. For each $p\geq 1$, the 
$p$-spectral radius $\lambda^{(p)}(G)$ is defined as
\[
\lambda^{(p)}(G):=\max_{|x_1|^p+\cdots+|x_n|^p=1}
r\sum_{\{i_1,\ldots,i_r\}\in E(G)}x_{i_1}\cdots x_{i_r}.
\]
The $p$-spectral radius was introduced by Keevash-Lenz-Mubayi, and subsequently 
studied by Nikiforov in 2014. The most extensively studied case is when $p=r$, and 
$\lambda^{(r)}(G)$ is called the spectral radius of $G$. The $\alpha$-normal 
labeling method, which was introduced by Lu and Man in 2014, is effective method 
for computing the spectral radii of uniform hypergraphs. It labels each corner 
of an edge by a positive number so that the sum of the corner labels at any vertex 
is $1$ while the product of all corner labels at any edge is $\alpha$. Since then, 
this method has been used by many researchers in studying $\lambda^{(r)}(G)$. In 
this paper, we extend Lu and Man's $\alpha$-normal labeling method to the $p$-spectral 
radii of uniform hypergraphs for $p\ne r$; and find some applications.\par\vspace{2mm}
  
 \par\vspace{2mm}

\end{abstract}

\section{Introduction}
\label{sec1}
Let $\mathbb{R}$ be the field of real numbers and $\mathbb{R}^n$ the $n$-dimensional 
real space. Given a vector $\bm{x}=(x_1,x_2,\ldots,x_n)^{\mathrm{T}}$ and a real 
number $p\geq 1$, we denote $||\bm{x}||_p:=(|x_1|^p+|x_2|^p+\cdots+|x_n|^p)^{1/p}$. 
We also denote $\mathbb{S}_p^{n-1}$ ($\mathbb{S}_{p,+}^{n-1}$, $\mathbb{S}_{p,++}^{n-1}$) 
the set of all (nonnegative, positive) real vectors $\bm{x}\in\mathbb{R}^n$ with 
$||\bm{x}||_p=1$. 

Let $G$ be an $r$-uniform hypergraph of order $n$, the polynomial form of $G$ is a 
multi-linear function $P_G(\bm{x}): \mathbb{R}^n\to\mathbb{R}$ defined for any vector 
$\bm{x}\in\mathbb{R}^n$ as
\[
P_G(\bm{x})=r\sum_{\{i_1,i_2,\ldots,i_r\}\in E(G)}x_{i_1}x_{i_2}\cdots x_{i_r}.
\]
For any real number $p\geq 1$, the {\em $p$-spectral radius} of $G$ is defined as
\[
\lambda^{(p)}(G):=\max_{||\bm{x}||_p=1}P_G(\bm{x}).
\]
If $\bm{x}\in\mathbb{S}^{n-1}_p$ is a vector such that $\lambda^{(p)}(G)=P_G(\bm{x})$, 
then $\bm{x}$ is called an {\em eigenvector} to $\lambda^{(p)}(G)$. Note that $P_G(\bm{x})$ 
can always reach its maximum at some nonnegative vectors. By Lagrange's method, we have 
the {\em eigenequation} for $\lambda^{(p)}(G)$ and $\bm{x}\in\mathbb{S}_{p,+}^{n-1}$ as follows:
\begin{equation}\label{eq:Eigenequation}
\sum_{\{i,i_2,\ldots,i_r\}\in E(G)}x_{i_2}\cdots x_{i_r}=\lambda^{(p)}(G)x_i^{p-1}
~~\text{for}~x_i\neq 0.
\end{equation}

The $p$-spectral radius has been introduced by Keevash, Lenz and Mubayi 
\cite{Keevash2014} and subsequently studied by Nikiforov 
\cite{Nikiforov2014:Analytic Methods,Nikiforov2014:Extremal Problems,Kang2014}
and Chang et al. \cite{ChangDing2017}. Note that the $p$-spectral radius 
$\lambda^{(p)}(G)$ shows remarkable connections with some hypergraph invariants. 
For instance, $\lambda^{(1)}(G)/r$ is equal to the Lagrangian of $G$, which 
has been investigated by Talbot \cite{Talbot2002}, $\lambda^{(r)}(G)$ is 
the usual spectral radius introduced by Cooper and Dutle \cite{Cooper2012}, 
and $\lambda^{(\infty)}(G)/r$ is the number of edges of $G$ (see 
\cite{Nikiforov2014:Analytic Methods}). It should be announced that we 
modified the definition of $p$-spectral radius by removing a constant 
factor $(r-1)!$ from \cite{Keevash2014}, so that the $p$-spectral radius 
is the same as the one in \cite{Cooper2012} when $p=r$. This is not essential 
and does not affect the results at all.

Recall that a {\em weighted incidence matrix} $B=(B(v,e))$ of a hypergraph 
$G$ is a $|V|\times |E|$ matrix such that for any vertex $v$ and any edge $e$, 
the entry $B(v,e)>0$ if $v\in e$ and $B(v,e)=0$ if $v\notin e$. In 
\cite{LuMan2016:Small Spectral Radius}, Lu and Man introduced the $\alpha$-normal 
labeling method for computing the spectral radii of uniform hypergraphs as follows: 
\begin{theorem}[\cite{LuMan2016:Small Spectral Radius}]
\label{thm:uniform hypergraph normal label}
Let $G$ be a connected $r$-uniform hypergraph. Then the spectral radius of $G$ is
$\rho(G)$ if and only if there is a weighted incidence matrix $B$ satisfying
\begin{enumerate}
\item[$(1)$] $\sum_{e:\,v\in e}B(v,e)=1$, for any $v\in V(G)$;
\item[$(2)$] $\prod_{v\in e}B(v,e)=\alpha=\rho(G)^{-r}$, for any $e\in E(G)$;
\item[$(3)$] $\prod_{i=1}^{\ell}\frac{B(v_{i-1},e_i)}{B(v_i,e_i)}=1$, for any cycle
$v_0e_1v_1e_2\cdots v_{\ell-1}e_{\ell}(v_{\ell}=v_0)$.
\end{enumerate}
\end{theorem}
The weighted incidence matrix $B(v,e)$ can be viewed as a labeling on the corners 
of edges. This $\alpha$-normal labeling method has been proved by many researches 
\cite{KangLiu2016,LiuKang2016,LiShao2016, BaiLu2017,OuyangQi2017,YuanSi2017,ZhangKang2017,XiaoWang2018} 
to be a simple and effective method in the study of spectral radii of uniform 
hypergraphs. In this paper, we extend Lu and Man's method to the $p$-spectral 
radii of uniform hypergraphs for $p\not=r$. The $\alpha$-normal labeling method 
(for $p\not=r$) is very different from the CSRH algorithm developed by Chang-Ding-Qi-Yan 
\cite{ChangDing2017} to compute the $p$-spectrum radii of uniform hypergraphs numerically.
Although our method can also be used to compute the $p$-spectral radius of a hypergraph 
$G$ when $G$ is highly symmetric or hypertree-like, this is not our main purpose. The 
goal of this paper is to provide a powerful tool to analyze the properties of the 
$p$-spectral radii of hypergraphs in the same way for the special case $p=r$. We 
illustrate this by giving several interesting applications.  We discover
two new monotonic functions characterizing the growth rate of $\lambda^{(p)}(G)$
with respective to the maximum degree and the minimum degrees
(\autoref{monotonic function}). We also prove two convex results (\autoref{convex} 
and \autoref{convex2}) of the $p$-spectral radius. We obtain a tight upper bound 
using degrees (\autoref{degree_bound}). We determine the $p$-spectral radius of 
$G_1\ast G_2$ (\autoref{thm:G1*G2}) and $G_1\times G_2$ (\autoref{product}). We 
study the $p$-spectral radius of the extension of a hypergraph (\autoref{extension}). 

The paper is organized in the following way: in Section 2, we develop the $\alpha$-normal 
labeling method for $p>r$. In Section 3, we present many applications. The $\alpha$-normal 
labeling method for $p<r$ is handled the last section.

\section{The $\alpha$-normal labeling method for $p>r$}

In this section, we will establish a relation between $\lambda^{(p)}(G)$ and its 
weighted incidence matrix of a uniform hypergraph $G$. For concepts on hypergraphs 
we refer the reader to \cite{Bretto2013}. Before continuing, we need the following 
Perron--Frobenius theorem for uniform hypergraphs.

\begin{theorem}[\cite{Nikiforov2014:Analytic Methods}]\label{thm:p>r}
Let $G$ be an $r$-uniform hypergraph with no isolated vertices. If $p>r$, then there 
exists a unique positive eigenvector to $\lambda^{(p)}(G)$.
\end{theorem}

Given an $r$-uniform hypergraph $G$, for each edge $e\in E(G)$, we put a weight $w(e)>0$ 
on $e$. In the following, we always assume that $p>r$.

\begin{definition}\label{def:consistent normal}
An $r$-uniform hypergraph $G$ is called {\em $\alpha$-normal} if there exist a weighted 
incidence matrix $B$ and weights $\{w(e)\}$ satisfying
\begin{enumerate}
\item[(1)] $\displaystyle\sum_{e\in E(G)}w(e)=1$; 

\item[(2)] $\displaystyle\sum_{e:\,v\in e}B(v,e)=1$, for any $v\in V(G)$; 
 
\item[(3)] $\displaystyle w(e)^{p-r}\cdot\prod_{v\in e}B(v,e)=\alpha$, 
for any $e\in E(G)$.
\end{enumerate}
Moreover, the weighted incidence matrix $B$ and weights $\{w(e)\}$ are called 
{\em consistent} if for any $v\in V(G)$ and $v\in e_i$, $i=1,2,\ldots,d$,
\[
\frac{w(e_1)}{B(v,e_1)}=\frac{w(e_2)}{B(v,e_2)}=\cdots=\frac{w(e_d)}{B(v,e_d)}.
\]
\end{definition}

\begin{lemma}\label{lem:iff}
Let $G$ be an $r$-uniform hypergraph with no isolated vertices. Then the $p$-spectral 
radius of $G$ is $\lambda^{(p)}(G)$ if and only if $G$ is consistently $\alpha$-normal 
with $\alpha=r^{p-r}/(\lambda^{(p)}(G))^p$.
\end{lemma}

\begin{proof}
($\Longrightarrow$) By \autoref{thm:p>r}, let 
$\bm{x}=(x_1,x_2,\ldots,x_n)^{\mathrm{T}}\in\mathbb{S}^{n-1}_{p,++}$ 
be an eigenvector to $\lambda^{(p)}(G)$. Define a weighted incidence 
matrix $B$ and $\{w(e)\}$ as follows:
\begin{align*}
B(v,e) & =\begin{dcases}
\displaystyle\frac{\prod_{u\in e}x_u}{\lambda^{(p)}(G) x_v^p}, & \text{if}~v\in e,\\
0, & \text{otherwise},
\end{dcases}\\[2mm]
w(e) & =\frac{r\prod_{u\in e}x_u}{\lambda^{(p)}(G)}.
\end{align*}
For any $v\in V(G)$, using the eigenequation \eqref{eq:Eigenequation} gives
\[
\sum_{e:\,v\in e}B(v,e)=
\frac{\sum_{e:\,v\in e}\prod_{u\in e}x_u}{\lambda^{(p)}(G) x_v^p}=1.
\]
Also, we see that
\[
\sum_{e\in E(G)}w(e)=\frac{r}{\lambda^{(p)}(G)}\sum_{e\in E(G)}\prod_{u\in e}x_u
=\frac{\lambda^{(p)}(G)}{\lambda^{(p)}(G)}=1.
\]
Therefore items (1) and (2) of \autoref{def:consistent normal} are verified. For 
item (3), we check that 
\begin{align*}
w(e)^{p-r}\cdot\prod_{v\in e}B(v,e) & =
\Bigg(\frac{r}{\lambda^{(p)}(G)}\prod_{u\in e}x_u\Bigg)^{p-r}\cdot
\prod_{v\in e}\frac{\prod_{u\in e}x_u}{\lambda^{(p)}(G)x_v^p}\\
& =\frac{r^{p-r}}{(\lambda^{(p)}(G))^p}=\alpha.
\end{align*}
To show that $B$ is consistent, for any $v\in V(G)$ and $v\in e_i$, $i=1,2,\ldots,d$, 
we have
\[
\frac{w(e_1)}{B(v,e_1)}=\frac{w(e_2)}{B(v,e_2)}=\cdots
=\frac{w(e_d)}{B(v,e_d)}=rx_v^p.
\]
    
($\Longleftarrow$) Assume that $G$ is consistently $\alpha$-normal with weighted 
incident matrix $B$ and $\{w(e)\}$. For any nonnegative vector 
$\bm{x}=(x_1,x_2,\ldots,x_n)^{\mathrm{T}}\in\mathbb{S}^{n-1}_{p,+}$, 
by H\"older's inequality and AM--GM inequality, we have
\begin{align*}
P_G(\bm{x}) 
& =r\sum_{\{i_1,i_2,\ldots,i_r\}\in E(G)}x_{i_1}x_{x_2}\cdots x_{i_r}\\
& =\frac{r}{\alpha^{1/p}}\sum_{e\in E(G)}\Bigg(w(e)^{(p-r)/p}\cdot\prod_{v\in e}\big(B(v,e)\big)^{1/p}x_v\Bigg)\\
& \leq \frac{r}{\alpha^{1/p}}\Bigg(\sum_{e\in E(G)}w(e)\Bigg)^{(p-r)/p}
  \Bigg(\sum_{e\in E(G)}\prod_{v\in e}\big(B(v,e)\big)^{1/r}x_v^{p/r}\Bigg)^{r/p}\\
& =\frac{r}{\alpha^{1/p}}\Bigg(\sum_{e\in E(G)}\prod_{v\in e}\big(B(v,e)\big)^{1/r}x_v^{p/r}\Bigg)^{r/p}\\
& \leq \frac{r^{1-r/p}}{\alpha^{1/p}}\Bigg(\sum_{e\in E(G)}\sum_{v\in e}B(v,e)x_v^p\Bigg)^{r/p}\\
& =\frac{r^{1-r/p}}{\alpha^{1/p}}\cdot||\bm{x}||_p^r=\frac{r^{1-r/p}}{\alpha^{1/p}}.
\end{align*}
This inequality implies $\lambda^{(p)}(G)\leq r^{1-r/p}\alpha^{-1/p}$.

The equality holds if $G$ is $\alpha$-normal and there is a nonzero solution $\bm{x}$ 
to the following equations:
\begin{equation}\label{eq:solution}
B(i_1,e)x_{i_1}^p=B(i_2,e)x_{i_2}^p=\cdots=B(i_r,e)x_{i_r}^p=\frac{w(e)}{r}, 
\end{equation}
for any $e=\{i_1,i_2,\ldots,i_r\}\in E(G)$. Define
\begin{equation}\label{eq:define x}
x_v^*=\left(\frac{w(e)}{rB(v,e)}\right)^{1/p},\ v\in e.   
\end{equation}
The consistent condition guarantees that $x_v^*$ is independent of the choice of the 
edge $e$. It is easy to check that $\bm{x}^*=(x_1^*,x_2^*,\ldots,x_n^*)^{\mathrm{T}}$ 
is a solution of \eqref{eq:solution}. Equation \eqref{eq:define x} also implies that
\[
rB(v,e)(x_v^*)^p=w(e),\ v\in e.   
\] 
Hence, the $\ell^p$-norm of $\bm{x}^*$ is
\begin{align*}
\sum_{v\in V(G)}(x_v^*)^p & =\sum_{e\in E(G)}\sum_{u\in e}B(u,e)(x_u^*)^p\\
& =\sum_{e\in E(G)}rB(u,e)(x_u^*)^p\\
& =\sum_{e\in E(G)}w(e)=1,
\end{align*}
which follows that $\lambda^{(p)}(G)=r^{1-r/p}\alpha^{-1/p}$.
\end{proof}

\begin{example}
Consider the following grid hypergraph $G_1$, which is a $4$-uniform hypergraph 
with $25$ vertices and $16$ edges generated by subdividing a square. Let 
\[
w_1=\frac{1}{4(1+4^{1/(p-2)})^2},~
w_2=\frac{4^{1/(p-2)}}{4(1+4^{1/(p-2)})^2},~
w_3=\frac{4^{2/(p-2)}}{4(1+4^{1/(p-2)})^2}.    
\]
For each vertex $v\in V(G_1)$ and edge $e\in E(G_1)$ with $v\in e$, we put a 
weight $w(e)$ at the center of $e$, and label the value $B(v,e)$ at each corner 
of the edge $e$ as follows:
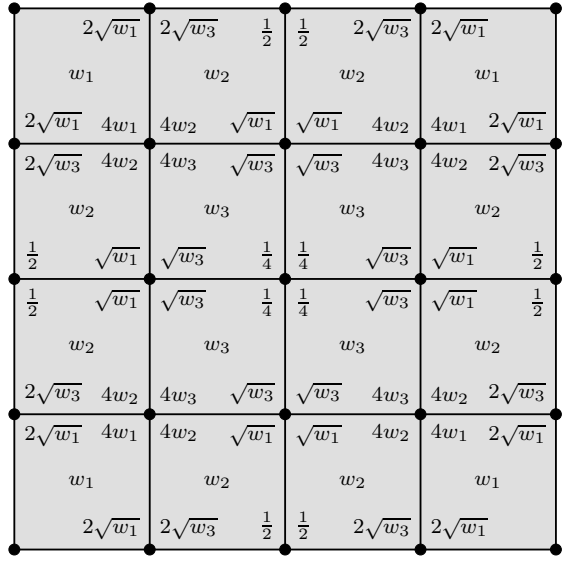
\begin{center}
\begin{tikzpicture}[scale=0.9,font=\scriptsize] 
\filldraw[fill=gray!25,draw=black,line width=0.7pt] (0,0) rectangle (8,8);  
\foreach \i in {2,4,6}
{
\draw[line width=0.7pt] (\i,0)--(\i,8);
\draw[line width=0.7pt] (0,\i)--(8,\i);
}

\foreach \v in {(1,1),(1,7),(7,1),(7,7)}
\node at \v {$w_1$};
\foreach \v in {(1,3),(1,5),(3,1),(3,7),(5,1),(5,7),(7,3),(7,5)}
\node at \v {$w_2$};
\foreach \v in {(3,3),(3,5),(5,3),(5,5)}
\node at \v {$w_3$};

\node[anchor=north east] at (2,2) {$4w_1$};
\node[anchor=north west] at (0,2) {$2\sqrt{w_1}$};
\node[anchor=south east] at (2,0) {$2\sqrt{w_1}$};
\node[anchor=south west] at (2,0) {$2\sqrt{w_3}$};
\node[anchor=south east] at (4,0) {$\frac12$};
\node[anchor=north east] at (4,2) {$\sqrt{w_1}$};
\node[anchor=north west] at (2,2) {$4w_2$};
\node[anchor=south west] at (4,0) {$\frac12$};
\node[anchor=south east] at (6,0) {$2\sqrt{w_3}$};
\node[anchor=north east] at (6,2) {$4w_2$};
\node[anchor=north west] at (4,2) {$\sqrt{w_1}$};
\node[anchor=south west] at (6,0) {$2\sqrt{w_1}$};
\node[anchor=north east] at (8,2) {$2\sqrt{w_1}$};
\node[anchor=north west] at (6,2) {$4w_1$};

\node[anchor=south west] at (0,2) {$2\sqrt{w_3}$};
\node[anchor=south east] at (2,2) {$4w_2$};
\node[anchor=north east] at (2,4) {$\sqrt{w_1}$};
\node[anchor=north west] at (0,4) {$\frac12$};
\node[anchor=south west] at (2,2) {$4w_3$};
\node[anchor=south east] at (4,2) {$\sqrt{w_3}$};
\node[anchor=north east] at (4,4) {$\frac14$};
\node[anchor=north west] at (2,4) {$\sqrt{w_3}$};
\node[anchor=south west] at (4,2) {$\sqrt{w_3}$};
\node[anchor=south east] at (6,2) {$4w_3$};
\node[anchor=north east] at (6,4) {$\sqrt{w_3}$};
\node[anchor=north west] at (4,4) {$\frac14$};
\node[anchor=south west] at (6,2) {$4w_2$};
\node[anchor=south east] at (8,2) {$2\sqrt{w_3}$};
\node[anchor=north east] at (8,4) {$\frac12$};
\node[anchor=north west] at (6,4) {$\sqrt{w_1}$};

\node[anchor=south west] at (0,4) {$\frac12$};
\node[anchor=south east] at (2,4) {$\sqrt{w_1}$};
\node[anchor=north east] at (2,6) {$4w_2$};
\node[anchor=north west] at (0,6) {$2\sqrt{w_3}$};
\node[anchor=south west] at (2,4) {$\sqrt{w_3}$};
\node[anchor=south east] at (4,4) {$\frac14$};
\node[anchor=north east] at (4,6) {$\sqrt{w_3}$};
\node[anchor=north west] at (2,6) {$4w_3$};
\node[anchor=south west] at (4,4) {$\frac14$};
\node[anchor=south east] at (6,4) {$\sqrt{w_3}$};
\node[anchor=north east] at (6,6) {$4w_3$};
\node[anchor=north west] at (4,6) {$\sqrt{w_3}$};
\node[anchor=south west] at (6,4) {$\sqrt{w_1}$};
\node[anchor=south east] at (8,4) {$\frac12$};
\node[anchor=north east] at (8,6) {$2\sqrt{w_3}$};
\node[anchor=north west] at (6,6) {$4w_2$};

\node[anchor=south west] at (0,6) {$2\sqrt{w_1}$};
\node[anchor=south east] at (2,6) {$4w_1$};
\node[anchor=north east] at (2,8) {$2\sqrt{w_1}$};
\node[anchor=south west] at (2,6) {$4w_2$};
\node[anchor=south east] at (4,6) {$\sqrt{w_1}$};
\node[anchor=north east] at (4,8) {$\frac12$};
\node[anchor=north west] at (2,8) {$2\sqrt{w_3}$};
\node[anchor=south west] at (4,6) {$\sqrt{w_1}$};
\node[anchor=south east] at (6,6) {$4w_2$};
\node[anchor=north east] at (6,8) {$2\sqrt{w_3}$};
\node[anchor=north west] at (4,8) {$\frac12$};
\node[anchor=south west] at (6,6) {$4w_1$};
\node[anchor=south east] at (8,6) {$2\sqrt{w_1}$};
\node[anchor=north west] at (6,8) {$2\sqrt{w_1}$};

\foreach \i in {0,2,4,6,8}
\foreach \j in {0,2,4,6,8}
\filldraw[fill=black] (\i,\j) circle (0.08);
\end{tikzpicture}
\end{center}
It can be checked that the grid hypergraph $G_1$ is consistently $\alpha$-normal with 
\[
\alpha=\frac{1}{4^{p-4}(1+4^{1/(p-2)})^{2(p-2)}}.   
\]
Therefore, the $p$-spectral radius of $G_1$ is 
\[
\lambda^{(p)}(G_1)=(16)^{1-4/p}(1+4^{1/(p-2)})^{2(p-2)/p}.    
\]
In \cite{Nikiforov2014:Analytic Methods}, Nikiforov proved that the $p$-spectral radius
is a Lipschitz function in $p$. Taking $p\to 4^+$, we obtain that the spectral radius 
of $G_1$ is $\rho(G_1)=3$.
\end{example}

The consistent condition in \autoref{def:consistent normal} shows that 
\[
B(v,e)=\frac{w(e)}{\sum_{f:\,v\in f}w(f)}
\] 
for any $v\in e$. Therefore we immediately have the following statement:
Let $G$ be an $r$-uniform hypergraph, then the $p$-spectral radius of $G$ is
$\lambda^{(p)}(G)$ if and only if there exist weights $\{w(e)\}$ such that 
$\sum_{e\in E(G)}w(e)=1$ and for each $e\in E(G)$,
\begin{equation}\label{eq:w}
\frac{w(e)^p}{\prod_{v\in e}\sum_{f:\,v\in f}w(f)}=\alpha,      
\end{equation}
with $\alpha=r^{p-r}/(\lambda^{(p)}(G))^p$. In some cases, the above conclusion is 
convenient to calculate $\lambda^{(p)}(G)$.

\begin{example}
Consider the following $3$-uniform hypergraph $G_2$ with $8$ vertices and $4$ edges.
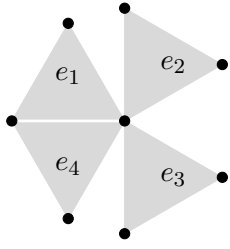
\begin{center}
\begin{tikzpicture}[scale=1.5]
\filldraw[fill=gray!30,draw=gray!30] (0,0)--(120:1)--(-1,0)--(0,0);
\filldraw[fill=gray!30,draw=gray!30] (0,0)--(-1,0)--(240:1)--(0,0);
\filldraw[fill=gray!30,draw=gray!30] (0,0)--(30:1)--(0,1)--(0,0);
\filldraw[fill=gray!30,draw=gray!30] (0,0)--(330:1)--(0,-1)--(0,0);
\draw[line width=1pt,white] (0,0)--(-1,0);

\foreach \u in {(0,0),(30:1),(0,1),(120:1),(-1,0),(240:1),(0,-1),(330:1)}
\filldraw[fill=black] \u circle (0.045);

\node at (-0.5,0.4) {$e_1$};
\node at (-0.5,-0.4) {$e_4$};
\node at (0.433,0.5) {$e_2$};
\node at (0.433,-0.5) {$e_3$};
\end{tikzpicture}  
\end{center}   
Putting a weight $w_i$ on each edge $e_i$, $i=1$, $2$, $3$, $4$, the consistent condition 
shows that $w_1=w_4$, and $w_2=w_3$. We can obtain $\lambda^{(p)}(G_2)$ by solving the 
following system of equations:
\[
\begin{dcases}
  \frac{w_1^p}{2w_1^2(2w_1+2w_2)}=\alpha\\
  \frac{w_2^p}{w_2^2(2w_1+2w_2)}=\alpha\\
 2w_1+2w_2=1.   
\end{dcases}    
\]
\end{example}
We get $w_1=\frac{2^{1/(p-2)}}{2(1+2^{1/(p-2)})}$, $w_2=\frac{1}{2(1+2^{1/(p-2)})}$, and
$\alpha=2^{2-p}(1+2^{1/(p-2)})^{2-p}$. Thus, the $p$-spectral radius of $G_2$ is
\[ 
\lambda^{(p)}(G_2)=3^{1-3/p}\cdot 2^{1-2/p}\cdot (1+2^{1/(p-2)})^{1-2/p}.
\]
In particular, taking $p\to 3^+$, we get $\rho(G_2)=\sqrt[3]{6}$.

\begin{definition}\label{def:subnormal}
An $r$-uniform hypergraph $G$ is called {\em $\alpha$-subnormal} if there exist a 
weighted incidence matrix $B$ and weights $\{w(e)\}$ satisfying
\begin{enumerate}
\item[(1)] $\displaystyle\sum_{e\in E(G)}w(e)\leq 1$;\\

\item[(2)] $\displaystyle\sum_{e:\,v\in e}B(v,e)\leq 1$, for any $v\in V(G)$;\\

\item[(3)] $\displaystyle w(e)^{p-r}\cdot\prod_{v\in e}B(v,e)\geq\alpha$, 
for any $e\in E(G)$.
\end{enumerate}
Moreover, $G$ is called {\em strictly $\alpha$-subnormal} if it is $\alpha$-subnormal 
but not $\alpha$-normal.
\end{definition}

\begin{lemma}\label{lem:subnormal}
Let $G$ be an $r$-uniform hypergraph. If $G$ is $\alpha$-subnormal, then the $p$-spectral
radius of $G$ satisfies
\[
\lambda^{(p)}(G)\leq \frac{r^{1-r/p}}{\alpha^{1/p}}.   
\]
\end{lemma}

\begin{proof}
For any nonnegative vector $\bm{x}=(x_1,x_2,\ldots,x_n)^{\mathrm{T}}\in\mathbb{S}^{n-1}_{p,+}$, 
by H\"older's inequality and AM--GM inequality, we have
\begin{align*}
r\sum_{\{i_1,\ldots,i_r\}\in E(G)}x_{i_1}\cdots x_{i_r} & \leq 
\frac{r}{\alpha^{1/p}}\sum_{e\in E(G)}\Bigg(w(e)^{(p-r)/p}\cdot\prod_{v\in e}\big(B(v,e)\big)^{1/p}x_v\Bigg)\\
& \leq \frac{r}{\alpha^{1/p}}\Bigg(\sum_{e\in E(G)}w(e)\Bigg)^{(p-r)/p}
\Bigg(\sum_{e\in E(G)}\prod_{v\in e}\big(B(v,e)\big)^{1/r}x_v^{p/r}\Bigg)^{r/p}\\
& \leq \frac{r}{\alpha^{1/p}}
\Bigg(\sum_{e\in E(G)}\prod_{v\in e}\big(B(v,e)\big)^{1/r}x_v^{p/r}\Bigg)^{r/p}\\
& \leq \frac{r^{1-r/p}}{\alpha^{1/p}}\Bigg(\sum_{e\in E(G)}\sum_{v\in e}B(v,e)x_v^p\Bigg)^{r/p}\\
& \leq \frac{r^{1-r/p}}{\alpha^{1/p}},
\end{align*}
which yields $\lambda^{(p)}(G)\leq r^{1-r/p}\alpha^{-1/p}$. When $G$ is strictly $\alpha$-subnormal, 
this inequality is strict, and therefore $\lambda^{(p)}(G)<r^{1-r/p}\alpha^{-1/p}$.
\end{proof}

\begin{definition}
An $r$-uniform hypergraph $G$ is called {\em $\alpha$-supernormal} if there exist a weighted 
incidence matrix $B$ and weights $\{w(e)\}$ satisfying
\begin{enumerate}
\item[(1)] $\displaystyle\sum_{e\in E(G)}w(e)\geq 1$;

\item[(2)] $\displaystyle\sum_{e:\,v\in e}B(v,e)\geq 1$, for any $v\in V(G)$;\\

\item[(3)] $\displaystyle w(e)^{p-r}\cdot\prod_{v\in e}B(v,e)\leq\alpha$, for any $e\in E(G)$.
\end{enumerate}
Moreover, $G$ is called {\em strictly $\alpha$-supernormal} if it is $\alpha$-supernormal but 
not $\alpha$-normal.
\end{definition}

\begin{lemma}\label{lem:supernormal}
Let $G$ be an $r$-uniform hypergraph. If $G$ is consistently $\alpha$-supernormal, then the 
$p$-spectral radius of $G$ satisfies
\[
\lambda^{(p)}(G)\geq \frac{r^{1-r/p}}{\alpha^{1/p}}.    
\]
\end{lemma}

\begin{proof}
The consistent condition implies that there exists a vector 
$\bm{x}=(x_1,x_2,\ldots,x_n)^{\mathrm{T}}\in\mathbb{S}_{p,++}^{n-1}$
satisfying \eqref{eq:solution}. Therefore
\begin{align*}
r\sum_{\{i_1,\ldots,i_r\}\in E(G)}x_{i_1}\cdots x_{i_r} & \geq 
\frac{r}{\alpha^{1/p}}\sum_{e\in E(G)}\Bigg(w(e)^{(p-r)/p}\cdot\prod_{v\in e}(B(v,e))^{1/p}x_v\Bigg)\\
& =\frac{r}{\alpha^{1/p}}\Bigg(\sum_{e\in E(G)}w(e)\Bigg)^{(p-r)/p}
  \Bigg(\sum_{e\in E(G)}\prod_{v\in e}\big(B(v,e)\big)^{1/r}x_v^{p/r}\Bigg)^{r/p}\\
& \geq \frac{r}{\alpha^{1/p}}\Bigg(\sum_{e\in E(G)}\prod_{v\in e}\big(B(v,e)\big)^{1/r}x_v^{p/r}\Bigg)^{r/p}\\
& =\frac{r^{1-r/p}}{\alpha^{1/p}}\Bigg(\sum_{e\in E(G)}\sum_{v\in e}B(v,e)x_v^p\Bigg)^{r/p}\\
& \geq \frac{r^{1-r/p}}{\alpha^{1/p}},
\end{align*}
which yields $\lambda^{(p)}(G)\geq r^{1-r/p}\alpha^{-1/p}$. When $G$ is strictly 
$\alpha$-supernormal, this inequality is strict, and therefore 
$\lambda^{(p)}(G)>r^{1-r/p}\alpha^{-1/p}$.
\end{proof}

\section{Applications for $p>r$}

In this section, we give some applications of $\alpha$-normal labeling method for the range
$p>r$. Let 
$G$ be an $r$-uniform hypergraph, and $G_i$ be the connected components of $G$, 
$i\in[k]$. If $1\leq p\leq r$, Nikiforov \cite{Nikiforov2014:Analytic Methods} 
proved that $\lambda^{(p)}(G)=\max_{1\leq i\leq k}\{\lambda^{(p)}(G_i)\}$, 
while the statement is different for $p>r$. Here we use \autoref{lem:iff} 
to give a new proof for the case $p>r$.

\begin{proposition}[\cite{Nikiforov2014:Analytic Methods}] 
Let $p>r\geq 2$, and let $G_1,G_2,\ldots,G_k$ be the components of an $r$-uniform 
hypergraph $G$. If $G$ has no isolated vertices, then
\[
\lambda^{(p)}(G)=\Bigg(\sum_{i=1}^k\big(\lambda^{(p)}(G_i)\big)^{p/(p-r)}\Bigg)^{(p-r)/p}.    
\]
\end{proposition}

\begin{proof}
For any $i\in[k]$, let $G_i$ be consistently $\alpha_i$-normal with weighted incidence 
matrix $B_i$ and $\{w_i(e)\}$, where $\alpha_i=r^{p-r}/(\lambda^{(p)}(G_i))^p$. That is
\[
\begin{dcases}
\sum_{e\in E(G_i)}w_i(e)=1,\\
\sum_{e\in E(G_i):\,v\in e}B_i(v,e)=1,\ \text{for any}\ v\in V(G_i),\\
w_i(e)^{p-r}\prod_{v\in e}B_i(v,e)=\alpha_i,\ \text{for any}\ e\in E(G_i).
\end{dcases}
\]
For convenience, we denote 
\[
C:=\sum_{i=1}^k\frac{1}{\alpha_i^{1/(p-r)}}=
\frac{1}{r}\Bigg(\sum_{i=1}^k\big(\lambda^{(p)}(G_i)\big)^{p/(p-r)}\Bigg).
\]
Now we construct a weighted incidence matrix $B$ and $\{w(e)\}$ for $G$ as follows:
\begin{align*}
B(v,e) & =\begin{cases}
B_i(v,e), & \text{if}\ v\in V(G_i), e\in E(G_i),\\
0, & \text{otherwise},
\end{cases}\\[1mm]
w(e) & =\frac{w_i(e)}{C\alpha_i^{1/(p-r)}},\ \text{if}\ e\in E(G_i). 
\end{align*}
For any $v\in V(G)$, assume that $v\in V(G_i)$, then
\[
\sum_{e\in E(G):\,v\in e}B(v,e)=\sum_{e\in E(G_i):\,v\in e}B_i(v,e)=1.    
\] 
For each edge $e\in E(G)$, assume that $e\in E(G_j)$, then
\[
w(e)^{p-r}\prod_{v\in e}B(v,e)=\frac{w_j(e)^{p-r}\prod_{v\in e}B_j(v,e)}{C^{p-r}\alpha_j}
=\frac{1}{C^{p-r}}.    
\]
Also we have
\[
\sum_{e\in E(G)}w(e)=\frac{1}{C}\sum_{i=1}^k\frac{\sum_{e\in E(G_i)}w_i(e)}{\alpha_i^{1/(p-r)}} 
=\frac{1}{C}\sum_{i=1}^k\frac{1}{\alpha_i^{1/(p-r)}}=1.   
\]
Clearly, $\{B(v,e)\}$ and $\{w(e)\}$ are consistent labeling of $G$. Therefore $G$ is 
consistently $\alpha$-normal with $\alpha=C^{-(p-r)}$. By \autoref{lem:iff} we obtain
\begin{align*}
\lambda^{(p)}(G) & =\Big(\frac{r^{p-r}}{\alpha}\Big)^{1/p}=(rC)^{(p-r)/p}\\
& =\Bigg(\sum_{i=1}^k\big(\lambda^{(p)}(G_i)\big)^{p/(p-r)}\Bigg)^{(p-r)/p},  
\end{align*}
completing the proof.
\end{proof}

\begin{theorem}\label{degree_bound}
Let $G$ be an $r$-uniform hypergraph, and $d_v$ be the degree of vertex $v$. If $p>r$, then
\[
\lambda^{(p)}(G)\leq\Bigg(r\sum_{e\in E(G)}\prod_{v\in e}d_v^{1/(p-r)}\Bigg)^{(p-r)/p}.    
\]   
\end{theorem}

\begin{proof}
For convenience, denote
\[
C:=\sum_{e\in E(G)}\prod_{v\in e}d_v^{1/(p-r)}.    
\]
Define a weighted incident matrix $B$ and $\{w(e)\}$ for $G$ as follows:
\begin{align*}
B(v,e) & =
    \begin{cases}
    1/d_v, & \text{if}\ v\in e,\\
    0,     & \text{otherwise},
    \end{cases}\\
w(e) & =\frac{1}{C}\prod_{v\in e}d_v^{1/(p-r)}. 
\end{align*}
It can be checked that $G$ is $\alpha$-subnormal with $\alpha=C^{-(p-r)}$.
By \autoref{lem:subnormal} we have
\[
\lambda^{(p)}(G)\leq\frac{r^{1-r/p}}{\alpha^{1/p}}=(rC)^{1-r/p}= 
\Bigg(r\sum_{e\in E(G)}\prod_{v\in e}d_v^{1/(p-r)}\Bigg)^{(p-r)/p}.   
\]
The proof is completed.
\end{proof}

From Theorem \ref{degree_bound}, we immediately obtain the following result.

\begin{corollary}\label{corollary}
Let $G$ be an $r$-uniform hypergraph with $m$ edges. If $p>r$, then
\[
\lambda^{(p)}(G)\leq (rm)^{1-r/p}\cdot\max_{e\in E(G)} 
\prod_{v\in e}d_v^{1/p}.   
\]
\end{corollary}

In the following, we present some properties of the function $\lambda^{(p)}(G)$ 
for a fixed $r$-uniform hypergraph $G$.

\begin{lemma}\label{lem:w(e) upper bound}
Suppose that $G$ is consistently $\alpha$-normal with weights $\{w(e)\}$. 
Let $\delta$ and $\Delta$ be the minimum degree and maximum degree of $G$, 
respectively. If $p>r$, then
\[
(\alpha\delta^r)^{1/(p-r)}\leq w(e)\leq (\alpha\Delta^r)^{1/(p-r)}.    
\] 
\end{lemma}

\begin{proof}
Without loss of generality, assume $w(e_1)=\min\{w(e): e\in E(G)\}$ and 
$w(e_2)=\max\{w(e): e\in E(G)\}$. According to \eqref{eq:w}, we have
\[
w(e_1)^p=\alpha\prod_{v\in e_1}\sum_{f:\,v\in f}w(f)\geq
\alpha(\delta\cdot w(e_1))^r,   
\] 
which yields $w(e_1)\geq (\alpha\delta^r)^{1/(p-r)}$. Similarly, we have
\[
w(e_2)^p=\alpha\prod_{v\in e_2}\sum_{f:\,v\in f}w(f)\leq
\alpha(\Delta\cdot w(e_2))^r,  
\]
which follows that $w(e_2)\leq (\alpha\Delta^r)^{1/(p-r)}$. The proof is completed. 
\end{proof}

\begin{theorem}\label{monotonic function}
Suppose that $G$ is an $r$-uniform hypergraph and $p>r$. Let
\[
f_G(p):=\bigg(\frac{\lambda^{(p)}(G)}{\Delta}\bigg)^{p/(p-r)},~
g_G(p):=\bigg(\frac{\lambda^{(p)}(G)}{\delta}\bigg)^{p/(p-r)}.    
\]    
Then $f_G(p)$ is non-decreasing in $p$ while $g_G(p)$ is non-increasing in $p$.
\end{theorem}

\begin{proof}
Let $G$ be consistently $\alpha$-normal with weighted incidence matrix $B$
and weights $\{w(e)\}$ for $\lambda^{(p)}(G)$, i.e.,
\[
\begin{dcases}
\sum_{e\in E(G)}w(e)=1,\\
\sum_{e:\,v\in e}B(v,e)=1,\ \text{for any}\ v\in V(G),\\
w(e)^{p-r}\prod_{v\in e}B(v,e)=\alpha,\ \text{for any}\ e\in E(G).
\end{dcases}
\]
Let $r<p<p'$. We now define a weighted incidence matrix $B'$ and $\{w'(e)\}$ 
for $\lambda^{(p')}(G)$ as follows: 
\[
B'(v,e)=B(v,e),~w'(e)=w(e).   
\]
Therefore, we obtain
\[
\sum_{e\in E(G)}w'(e)=\sum_{e\in E(G)}w(e)=1
\]
and for any $v\in V(G)$,
\[
\sum_{e\in E(G):\,v\in e}B'(v,e)=\sum_{e\in E(G):\,v\in e}B(v,e)=1.
\]
Using \autoref{lem:w(e) upper bound} gives
\begin{align*}
w'(e)^{p'-r}\prod_{v\in e}B'(v,e)=w(e)^{p'-p}\alpha
& \leq (\alpha\Delta^r)^{(p'-p)/(p-r)}\alpha\\
& =\alpha^{(p'-r)/(p-r)}\Delta^{r(p'-p)/(p-r)}\\
& =r^{p'-r}\bigg(\frac{\Delta^{r(p'-p)}}{(\lambda^{(p)}(G))^{p(p'-r)}}\bigg)^{1/(p-r)}.  
\end{align*}
Hence, $G$ is consistently $\alpha'$-supernormal for $\lambda^{(p')}(G)$ with 
\[
\alpha'=r^{p'-r}\bigg(\frac{\Delta^{r(p'-p)}}{(\lambda^{(p)}(G))^{p(p'-r)}}\bigg)^{1/(p-r)}.     
\]
By \autoref{lem:supernormal} we have
\[
\lambda^{(p')}(G)\geq\frac{r^{1-r/p'}}{(\alpha')^{1/p'}}=
\bigg(\frac{(\lambda^{(p)}(G))^{p(p'-r)}}{\Delta^{r(p'-p)}}\bigg)^{1/(p'(p-r))},    
\]
which implies that
\[
\bigg(\frac{\lambda^{(p')}(G)}{\Delta}\bigg)^{p'/(p'-r)}\geq 
\bigg(\frac{\lambda^{(p)}(G)}{\Delta}\bigg)^{p/(p-r)}.  
\]
Similarly, we can obtain
\[
\bigg(\frac{\lambda^{(p')}(G)}{\delta}\bigg)^{p'/(p'-r)}\leq 
\bigg(\frac{\lambda^{(p)}(G)}{\delta}\bigg)^{p/(p-r)}.  
\]
The proof is completed. 
\end{proof}
Assume that $G$ has $m$ edges.
In \cite{Nikiforov2014:Analytic Methods}, Nikiforov proved that the function
$(\lambda^{(p)}(G)/(rm))^p$ is non-increasing in $p$. Here we give 
a new proof for $p>r$ using $\alpha$-normal labeling method.

\begin{theorem}[\cite{Nikiforov2014:Analytic Methods}]
Let $G$ be an $r$-uniform hypergraph with $m$ edges. Then the function 
$(\lambda^{(p)}(G)/(rm))^p$ is non-increasing in $p$.     
\end{theorem}

\begin{proof}
Let $G$ be consistently $\alpha$-normal with weighted incidence matrix $B$ and weights 
$\{w(e)\}$ for $\lambda^{(p)}(G)$, where $\alpha=r^{p-r}/(\lambda^{(p)}(G))^p$.
Let $r<p<p'$. We define a weighted incidence matrix $B'$ and $\{w'(e)\}$ for 
$\lambda^{(p')}(G)$ as follows: 
\[
B'(v,e)=B(v,e),~w'(e)=\bigg(\frac{w(e)^{p-r}}{m^{p'-p}}\bigg)^{1/(p'-r)}.    
\] 
It is obvious that
\[
\sum_{e\in E(G):\,v\in e}B'(v,e)=\sum_{e\in E(G):\,v\in e}B(v,e)=1.    
\]
By H\"older's inequality, we see
\[
\sum_{e\in E(G)}w'(e)=\frac{1}{m^{(p'-p)/(p'-r)}}\sum_{e\in E(G)}w(e)^{(p-r)/(p'-r)}\leq 1.
\] 
For each edge $e\in E(G)$, we have
\[
w'(e)^{p'-r}\prod_{v\in e}B'(v,e)=\frac{w(e)^{p-r}}{m^{p'-p}}\prod_{v\in e}B(v,e)
=\frac{\alpha}{m^{p'-p}}.    
\]
Therefore, $G$ is $\alpha'$-subnormal for $\lambda^{(p')}(G)$, where $\alpha'=\alpha m^{p-p'}$. 
Using \autoref{lem:subnormal} gives 
\[
(\lambda^{(p')}(G))^{p'}\leq\frac{r^{p'-r}}{\alpha'}=(rm)^{p'-p}\cdot(\lambda^{(p)}(G))^p,    
\]
the result follows.
\end{proof}

\begin{theorem}\label{convex}
For any $r$-uniform hypergraph $G$, the function $h_G(p):=p\log \lambda^{(p)}(G)$ is concave 
upward on $(r,\infty)$.
\end{theorem}

\begin{proof}
For any $r<p_1<p<p_2$, write $p=\mu p_1 + (1-\mu)p_2$, where $\mu=(p_2-p)/(p_2-p_1)$.
According to \autoref{lem:iff}, let $G$ be consistently $\alpha_i$-normal with weighted 
incident matrix $B_i$ and $\{w_i(e)\}$ for $\lambda^{(p_i)}(G)$, where 
$\alpha_i=r^{p_i-r}/(\lambda^{(p_i)}(G))^{p_i}$, $i=1$, $2$. That is
\[
\begin{dcases}
\sum_{e\in E(G)}w_i(e)=1,\\
\sum_{e:\,v\in e}B_i(v,e)=1,\ \text{for any}\ v\in V(G),\\
w_i(e)^{p_i-r}\prod_{v\in e}B_i(v,e)=\alpha_i,\ \text{for any}\ e\in E(G).
\end{dcases}
\]

Let $\xi=\frac{(p_1-r)(p_2-p)}{(p-r)(p_2-p_1)}\in (0,1)$. We define a weighted 
incidence matrix $B$ and $\{w(e)\}$ for $\lambda^{(p)}(G)$ as follows:
\begin{align*}
  B(v,e) & =\mu B_1(v,e)+(1-\mu) B_2(v,e),\\
  w(e)   & =\xi w_1(e) +(1-\xi)w_2(e).
\end{align*}
In the following we shall prove $\{B(v,e)\}$ and $\{w(e)\}$ are $\alpha$-subnormal 
labeling for $\alpha=\alpha_1^\mu\alpha_2^{1-\mu}$ (and for $\lambda^{(p)}(G)$). 
For any vertex $v\in V(G)$,
\begin{align*}
\sum_{e\in E(G):\,v\in e} B(v,e) 
& =\mu \sum_{e\in E(G):\,v\in e} B_1(v,e) + (1-\mu)\sum_{e\in E(G):\,v\in e} B_2(v,e)\\
& =\mu+(1-\mu)=1.
\end{align*}
We also have
\[ 
\sum_{e\in E(G)}w(e)= \xi \sum_{e\in E(G)}w_1(e) + (1-\xi) \sum_{e\in E(G)}w_2(e)
=\xi+(1-\xi)=1.
\]
For any edge $e\in E(G)$, we have
\begin{align*}
w(e)^{p-r}\prod_{v\in e}B(v,e) 
& =[\xi w_1(e) +(1-\xi)w_2(e)]^{p-r} \prod_{v\in e}(\mu B_1(v,e)+(1-\mu) B_2(v,e))\\
& \geq w_1(e)^{\xi(p-r)}w_2(e)^{(1-\xi)(p-r)}\prod_{v\in e} (B_1(v,e))^\mu (B_2(v,e))^{1-\mu}\\
& =w_1(e)^{\mu (p_1-r)} w_2(e)^{(1-\mu)(p_2-r)}\prod_{v\in e} (B_1(v,e))^\mu (B_2(v,e))^{1-\mu}\\
& = \bigg(w_1(e)^{p_1-r}\prod_{v\in e} B_1(v,e)\bigg)^\mu
    \bigg(w_2(e)^{p_2-r}\prod_{v\in e} B_2(v,e)\bigg)^{1-\mu}\\
& =\alpha_1^\mu\alpha_2^{1-\mu}.
\end{align*}
By \autoref{lem:subnormal}, we have
\begin{align*}
p\log \lambda^{(p)}(G)
& \leq p\log (r^{1-r/p} \alpha^{-1/p})\\
& = (p-r)\log r - \mu \log \alpha_1 -(1-\mu) \log \alpha_2\\
& =\mu[(p_1-r)\log r-\log \alpha_1] + (1-\mu)[(p_2-r)\log r-\log \alpha_2]\\
& =\mu p_1\log \lambda^{(p_1)}(G) +(1-\mu) p_2\log \lambda^{(p_2)}(G).
\end{align*}
Thus the function $h_G(p)=p\log \lambda^{(p)}(G)$ is concave upward on $(r,\infty)$.
\end{proof}

\begin{corollary}
The function $\lambda^{(p)}(G)$ is differentiable on $(r,\infty)$ except 
countable many $p$'s. Moreover, the left-hand derivative and the right-hand 
derivative always exist for any $p>r$.
\end{corollary}

\begin{remark}
In \cite{Nikiforov2014:Analytic Methods}, Nikiforov provides an example showing 
that $\lambda^{(p)}(G)$ may not be differentiable at $p=2,3,\ldots, r$. He asked 
whether $\lambda^{(p)}(G)$ is continuously differentiable on $(r,\infty)$. This 
corollary give a weak solution toward this problem.
\end{remark}

\begin{theorem}\label{convex2}
For any $r$-uniform hypergraph $G$ and $p>r$, the function $\log \lambda^{(p)}(G)$ 
is concave upward in $1/p$.
\end{theorem}

\begin{proof}
For any $p_i>r$, $i=1$, $2$, write 
\[
\frac{1}{p}=\frac{\mu}{p_1} + \frac{1-\mu}{p_2},\ \text{where}\
\mu=\frac{p_1(p_2-p)}{p(p_2-p_1)}.
\]
Let $G$ be consistently $\alpha_i$-normal with weighted 
incident matrix $B_i$ and $\{w_i(e)\}$ for $\lambda^{(p_i)}(G)$, where 
$\alpha_i=r^{p_i-r}/(\lambda^{(p_i)}(G))^{p_i}$, $i=1$, $2$. Therefore
\[
\begin{dcases}
\sum_{e\in E(G)}w_i(e)=1,\\
\sum_{e:\,v\in e}B_i(v,e)=1,\ \text{for any}\ v\in V(G),\\
w_i(e)^{1-r/p_i}\prod_{v\in e}(B_i(v,e))^{1/p_i}=\alpha_i^{1/p_i},\ 
\text{for any}\ e\in E(G).
\end{dcases}
\]
Furthermore, let 
\[
\eta=\frac{p}{p_1}\mu,~\xi=\frac{p(p_1-r)}{p_1(p-r)}\mu.
\] 
We define a weighted incidence 
matrix $B$ and $\{w(e)\}$ for $\lambda^{(p)}(G)$ as follows:
\begin{align*}
  B(v,e) & =\eta B_1(v,e)+(1-\eta) B_2(v,e),\\
  w(e)   & =\xi w_1(e) +(1-\xi)w_2(e).
\end{align*}
In the following we shall prove $\{B(v,e)\}$ and $\{w(e)\}$ are 
$\alpha$-subnormal labeling for $\lambda^{(p)}(G)$, where 
$\alpha^{1/p}=\alpha_1^{\mu/p_1}\alpha_2^{(1-\mu)/p_2}$. 
It is clear that $\sum_{e:\,v\in e} B(v,e)=1$ for any 
$v\in V(G)$, and $\sum_{e\in E(G)}w(e)=1$. For any 
edge $e\in E(G)$, we have
\begin{align*}
\bigg(w(e)^{p-r}\prod_{v\in e}B(v,e)\bigg)^{1/p}
& \!\geq\! \bigg(w_1(e)^{\xi(p-r)}w_2(e)^{(1-\xi)(p-r)}\prod_{v\in e}%
  (B_1(v,e))^{\eta} (B_2(v,e))^{1-\eta}\bigg)^{1/p}\\
& \!=\! \bigg(\!w_1(e)^{p_1-r}\prod_{v\in e} B_1(v,e)\!\bigg)^{\mu/p_1}\!
    \bigg(\!w_2(e)^{p_2-r}\prod_{v\in e} B_2(v,e)\!\bigg)^{(1-\mu)/p_2}\\
& =\alpha_1^{\mu/p_1}\alpha_2^{(1-\mu)/p_2}.
\end{align*}
By \autoref{lem:subnormal}, we have
\begin{align*}
\log \lambda^{(p)}(G)
& \leq \bigg(1-\frac{r}{p}\bigg)\log r-\frac{1}{p}\log\alpha\\
& =\mu\log \lambda^{(p_1)}(G)+(1-\mu)\log\lambda^{(p_2)}(G).
\end{align*}
Thus the function $\log \lambda^{(p)}(G)$ is concave upward in $1/p$.
\end{proof}

Now we consider the applications of the $\alpha$-normal labeling method in the study 
of the products of hypergraphs. Let $G_i$ be $r_i$-uniform hypergraph, $i=1$, $2$, 
with $V(G_1)\cap V(G_2)=\emptyset$. Define an $(r_1+r_2)$-uniform hypergraph $G_1*G_2$ 
by
\[
V(G_1*G_2)=V(G_1)\cup V(G_2),~
E(G_1*G_2)=\{e\cup f\,|\,e\in E(G_1), f\in E(G_2)\}.   
\]
In \cite{Nikiforov2014:Analytic Methods}, Nikiforov investigated the $p$-spectral radius 
of $G_1*tK_1$ and $G_1*K_t^t$ for $p\geq 1$, where $K_t^t$ is a $t$-uniform hypergraph 
with only one edge. In the following, we give a generalized result for large $p$.

\begin{theorem}\label{thm:G1*G2}
Suppose that $G_i$ is an $r_i$-uniform hypergraph, $i=1$, $2$. If $p>r_1+r_2$, then
\[
\lambda^{(p)}(G_1* G_2)=\frac{(r_1+r_2)^{1-(r_1+r_2)/p}}{r_1^{1-r_1/p}r_2^{1-r_2/p}}
\lambda^{(p)}(G_1)\lambda^{(p)}(G_2).    
\]
\end{theorem}

\begin{proof}
According to \autoref{lem:iff}, let $G_i$ be consistently $\alpha_i$-normal with 
weighted incident matrix $B_i$ and $\{w_i(e)\}$ for $\lambda^{(p)}(G_i)$, where 
$\alpha_i=r_i^{p-r_i}/(\lambda^{(p)}(G_i))^p$, $i=1$, $2$. That is
\[
\begin{dcases}
\sum_{e\in E(G_i)}w_i(e)=1,\\
\sum_{e\in E(G_i):\,v\in e}B_i(v,e)=1,\ \text{for any}\ v\in V(G_i),\\
w_i(e)^{p-r_i}\prod_{v\in e}B_i(v,e)=\alpha_i,\ \text{for any}\ e\in E(G_i).
\end{dcases}
\]
For any $v\in V(G_1*G_2)$ and $e\cup f\in E(G_1*G_2)$, we define a weighted incidence 
matrix $B$ and $\{w(e\cup f)\}$ for $G_1*G_2$ as follows:
\begin{align}\label{eq:B(v,e+f)}
B(v,e\cup f) & =
    \begin{cases}
    B_1(v,e)\cdot w_2(f), & \text{if}\ v\in e,\\
    B_2(v,f)\cdot w_1(e), & \text{if}\ v\in f,\\
    0, & \text{otherwise}, 
    \end{cases}\\
w(e\cup f) & =w_1(e)\cdot w_2(f). \label{eq:c(e+f)}    
\end{align}
In the following we shall prove that $\{B(v,e\cup f)\}$ and $\{w(e\cup f)\}$ are
consistent $\alpha$-normal labeling of $G_1*G_2$ with $\alpha=\alpha_1\alpha_2$. 

(i). By \eqref{eq:c(e+f)} we have
\[
\sum_{e\,\cup f\in E(G_1*G_2)}w(e\cup f)=\Bigg(\sum_{e\in E(G_1)}w_1(e)\Bigg)
\Bigg(\sum_{f\in E(G_2)}w_2(f)\Bigg)=1.    
\]

(ii). For any $v\in V(G_1*G_2)$, if $v\in V(G_1)$ we have
\begin{align*}
\sum_{e\,\cup f:\,v\in e\,\cup f}B(v,e\cup f) & =\sum_{e\,\cup f:\,v\in e\,\cup f}B_1(v,e)w_2(f)\\
& =\sum_{e\in E(G_1):\,v\in e}\sum_{f\in E(G_2)}B_1(v,e)w_2(f)\\
& =\Bigg(\sum_{e\in E(G_1):\,v\in e}B_1(v,e)\Bigg)\Bigg(\sum_{f\in E(G_2)}w_2(f)\Bigg)\\
& =1.    
\end{align*}
If $v\in V(G_2)$, we can prove $\sum_{e\,\cup f:\,v\in e\,\cup f}B(v,e\cup f)=1$ similarly.

(iii). For any edge $e\cup f\in E(G_1*G_2)$, it follows from \eqref{eq:B(v,e+f)}
and \eqref{eq:c(e+f)} that
\begin{align*}
&~w(e\cup f)^{p-(r_1+r_2)}\cdot\prod_{v\in e\,\cup f}B(v,e\cup f)\\ 
= &~\big(w_1(e)w_2(f)\big)^{p-(r_1+r_2)}\cdot\prod_{v\in e}B(v,e\cup f)\prod_{u\in f}B(u,e\cup f)\\
= &~w_1(e)^{p-r_1}\prod_{v\in e}B_1(v,e)\cdot w_2(f)^{p-r_2}\prod_{u\in f}B_2(u,f)\\
= &~\alpha_1\alpha_2.
\end{align*}
Finally, consider each vertex $v\in V(G_1*G_2)$ and $v\in e\cup f\in E(G_1*G_2)$. 
By \eqref{eq:B(v,e+f)} and \eqref{eq:c(e+f)} we see that
\[
\frac{w(e\cup f)}{B(v,e\cup f)}=
\begin{dcases}
\frac{w_1(e)}{B_1(v,e)}, & \text{if}\ v\in V(G_1),\\[1mm]
\frac{w_2(f)}{B_2(v,f)}, & \text{if}\ v\in V(G_2).
\end{dcases}    
\]
Hence, $G_1*G_2$ is consistently $\alpha$-normal with $\alpha=\alpha_1\alpha_2$.
Using \autoref{lem:iff} gives
\begin{align*}
\lambda^{(p)}(G_1*G_2) & =\frac{(r_1+r_2)^{1-(r_1+r_2)/p}}{(\alpha_1\alpha_2)^{1/p}}\\
& =\frac{(r_1+r_2)^{1-(r_1+r_2)/p}}{r_1^{1-r_1/p}r_2^{1-r_2/p}}
\lambda^{(p)}(G_1)\lambda^{(p)}(G_2).     
\end{align*}
The proof is completed.
\end{proof}

Let $G_1$ and $G_2$ be two $r$-uniform hypergraphs. The {\em direct product} $G_1\times G_2$ 
of $G_1$ and $G_2$ is defined as $V(G_1\times G_2)=V(G_1)\times V(G_2)$, and 
$\{(i_1,j_1),\ldots,(i_r,j_r)\}\in E(G_1\times G_2)$ if and only if 
$\{i_1,\ldots,i_r\}\in E(G_1)$ and $\{j_1,\ldots,j_r\}\in E(G_2)$. For 
an edge $f=\{(i_1,j_1),\ldots,(i_r,j_r)\}\in E(G_1\times G_2)$, we denote
$\pi_1(f):=\{i_1,\ldots,i_r\}\in E(G_1)$ and $\pi_2(f):=\{j_1,\ldots,j_r\}\in E(G_2)$.

In what follows, we give an extension to a result of Shao \cite{Shao2013}.   

\begin{theorem} \label{product}
Let $G_1$ and $G_2$ be two $r$-uniform hypergraphs. If $p>r$, then
\[
\lambda^{(p)}(G_1\times G_2)=(r-1)!\lambda^{(p)}(G_1)\lambda^{(p)}(G_2).    
\]
\end{theorem}

\begin{proof}
By \autoref{lem:iff}, let $\{B_i(v,e)\}$ and $\{w_i(e)\}$ be the consistent 
$\alpha_i$-normal labeling of $G_i$ with $\alpha_i=r^{p-r}/(\lambda^{(p)}(G_i))^p$, 
$i=1$, $2$. 

Define a weighted incident matrix $B$ and $\{w(f)\}$ for $G_1\times G_2$ as follows:
\begin{align}\label{eq:B((u,v))}
B((u,v),f) & =
\begin{dcases}
\frac{B_1(u,\pi_1(f))B_2(v,\pi_2(f))}{(r-1)!}, & \text{if}\ (u,v)\in f,\\
0, & \text{otherwise},
\end{dcases}\\[2mm]\label{eq:w1w2}
w(f) & =\frac{w_1(\pi_1(f))w_2(\pi_2(f))}{r!}. 
\end{align}  
In what follows, we shall prove that $\{B((u,v),f)\}$ and $\{w(f)\}$ are consistent 
$\alpha$-normal labeling of $G_1\times G_2$ with $\alpha=r^r\alpha_1\alpha_2/(r!)^p$.

(i). By \eqref{eq:w1w2} we deduce that
\begin{align*}
\sum_{f\in E(G_1\times G_2)}w(f) 
& =\frac{1}{r!}\sum_{f\in E(G_1\times G_2)}w_1(\pi_1(f))w_2(\pi_2(f))\\
& =\sum_{e_1\in E(G_1)}\sum_{e_2\in E(G_2)}w_1(e_1)w_2(e_2)\\
& =1.    
\end{align*}

(ii). For any $(u,v)\in V(G_1\times G_2)$, we have
\begin{align*}
\sum_{f:\,(u,v)\in f}B((u,v),f) & =
\frac{1}{(r-1)!}\sum_{f:\,(u,v)\in f}B_1(u,\pi_1(f))B_2(v,\pi_2(f))\\
& =\sum_{e_1:\,u\in e_1}\sum_{e_2:\,v\in e_2}B_1(u,e_1)B_2(v,e_2)\\
& =1.    
\end{align*} 

(iii). For each edge $f\in E(G_1\times G_2)$, we obtain that
\begin{align*}
&~w(f)^{p-r}\prod_{(u,v)\in f}B((u,v),f)\\
= &~\frac{[w_1(\pi_1(f))w_2(\pi_2(f))]^{p-r}}{(r!)^{p-r}}\cdot
\prod_{(u,v)\in f}\frac{B_1(u,\pi_1(f))B_2(v,\pi_2(f))}{(r-1)!}\\
= &~\frac{r^r\alpha_1\alpha_2}{(r!)^p}.
\end{align*}
Finally, consider each vertex $(u,v)\in V(G_1)\times V(G_2)$ and 
$(u,v)\in f$. It follows from \eqref{eq:B((u,v))} and
\eqref{eq:w1w2} that
\[
\frac{w(f)}{B((u,v),f)}=
\frac{w_1(\pi_1(f))}{rB_1(u,\pi_1(f))}\cdot\frac{w_2(\pi_2(f))}{B_2(v,\pi_2(f))}.
\]
Therefore $G_1\times G_2$ is consistently $\alpha$-normal with
$\alpha=r^r\alpha_1\alpha_2/(r!)^p$. Hence
\[
\lambda^{(p)}(G_1\times G_2)=\frac{r^{1-r/p}}{\alpha^{1/p}}
=(r-1)!\lambda^{(p)}(G_1)\lambda^{(p)}(G_2).
\]
The proof is completed.
\end{proof}

For a given $r$-uniform hypergraph $G$, we let $G^{r+1}$ be an $(r+1)$-uniform 
hypergraph obtained by adding a new vertex $v_e$ in each edge $e$ of $G$ such 
that all these new vertices are pairwise disjoint. Following \cite{KangLiu2016}, 
$G^{r+1}$ is a {\em generalized power} of $G$. If we do not require that all $v_e$
to be distinct, we get an {\em extension} of $G$. Denote $\mathcal{E}(G)$ the set 
of all extensions of $G$. 

\begin{theorem} \label{extension}
Let $G$ be an $r$-uniform hypergraph with no isolated vertices, and $H$ be an 
extension of $G$. If $p>r$, then 
\[
\Big(\Big(\frac{r+1}{r}\Big)^{p-r}\big(\lambda^{(p)}(G)\big)^p\Big)^{1/(p+1)}
\leq\lambda^{(p+1)}(H)\leq\left(\frac{(r+1)^{p-r}}{r^{p+1-r}}\right)^{1/(p+1)}\lambda^{(p+1)}(G),   
\]
with the left equality holds if and only if $H\cong G^{r+1}$, and the right
equality holds if and only if $H\cong G*K_1$.
\end{theorem}

\begin{proof}
We first prove the left inequality. Let $\mathcal{E}_i(G)$ be the subset of 
$\mathcal{E}(G)$ in which each member $H$ has exactly $i$ vertices of degree 
one in $V(H)\backslash V(G)$. Clearly,
\[
\mathcal{E}(G)=\bigcup_{i=0}^{m}\mathcal{E}_i(G),    
\]
where $m$ is the size of $G$. Let $H\in\mathcal{E}_i(G)$. If $i=m$, then 
$H\cong G^{r+1}$. If $i\leq m-1$, we claim that there is an extension 
$H'\in\mathcal{E}_{i+1}(G)$ such that $\lambda^{(p+1)}(H')<\lambda^{(p+1)}(H)$. 
Choose a vertex $v_0\in V(H)\backslash V(G)$ with degree great than one, and an 
edge $e_0\in E(H)$ such that $v_0\in e_0$. Let $H'$ be the extension of $G$ with 
$V(H')=V(H)\cup\{u_0\}$ and $E(H')=(E(H)\backslash e_0)\cup ((e_0\backslash\{v_0\})\cup\{u_0\})$, 
where $u_0\notin V(H)$ is a new vertex. Assume $\bm{x}$ is the positive 
eigenvector with $||\bm{x}||_{p+1}=1$ corresponding to $\lambda^{(p+1)}(H')$, 
we define a vector $\bm{y}$ for $H$ as follows:
\[
y_v=\begin{cases}
    x_v, & v\neq v_0,\\[2mm]
    \sqrt[\uproot{5}p+1]{x_{u_0}^{p+1}+x_{v_0}^{p+1}}, & v=v_0.
    \end{cases}    
\]
It follows that
\begin{align*}
\lambda^{(p+1)}(H)-\lambda^{(p+1)}(H') 
& \geq r\sum_{e\in E(H)}\prod_{v\in e}y_v-r\sum_{e\in E(H')}\prod_{v\in e}x_v\\
& =r(y_{v_0}-x_{u_0})\prod_{v\in e_0\backslash\{v_0\}}x_v+
r(y_{v_0}-x_{v_0})\sum_{\substack{e\in E(H)\backslash e_0,\\ v_0\in e}}
\prod_{v\in e\backslash\{v_0\}}x_v\\
& >0,
\end{align*}
which yields $\lambda^{(p+1)}(H)>\lambda^{(p+1)}(H')$. Therefore
$\lambda^{(p+1)}(H)\geq\lambda^{(p+1)}(G^{r+1})$, with equality 
if and only if $H\cong G^{r+1}$. Now it suffices to show that
\[
\lambda^{(p+1)}(G^{r+1})=
\Big(\Big(\frac{r+1}{r}\Big)^{p-r}\big(\lambda^{(p)}(G)\big)^p\Big)^{1/(p+1)}.
\]
By \autoref{lem:iff}, let $G$ be consistently $\alpha$-normal with weighted 
incident matrix $B$ and $\{w(e)\}$, where $\alpha=r^{p-r}/(\lambda^{(p)}(G))^p$. 
That is
\[
\begin{dcases}
\sum_{e\in E(G)}w(e)=1,\\
\sum_{e:\,v\in e}B(v,e)=1,\ \text{for any}\ v\in V(G),\\
w(e)^{p-r}\prod_{v\in e}B(v,e)=\alpha,\ \text{for any}\ e\in E(G).
\end{dcases}
\]
We now define a weighted incident matrix $B'=(B'(v,e\cup\{v_e\}))$ and $\{w'(e\cup\{v_e\})\}$
for $G^{r+1}$ as follows:
\begin{align*}
B'(v,e\cup\{v_e\}) & =\begin{cases}
B(v,e), & \text{if}\ v\in e,\\
1, & \text{if}\ v=v_e,\\
0, & \text{otherwise},
\end{cases}\\
w'(e\cup\{v_e\}) & =w(e). 
\end{align*}
It can be checked that
\[
\sum_{e\,\cup\{v_e\}\in E(G^{r+1})}w'(e\cup\{v_e\})=1,~
\sum_{e\,\cup\{v_e\}:\,v\in e\,\cup\{v_e\}}B'(v,e\cup\{v_e\})=1,
\]
and for each edge $e\cup\{v_e\}\in E(G^{r+1})$, we have
\[
(w'(e\cup\{v_e\}))^{(p+1)-(r+1)}\prod_{v\in e\,\cup\{v_e\}}B'(v,e\cup\{v_e\})=\alpha.    
\]
Clearly, the weighted incidence matrix $B'$ and $\{w'(e\cup\{v_e\})\}$ are consistent. 
Using \autoref{lem:iff} gives
\[
\frac{r^{p-r}}{(\lambda^{(p)}(G))^p}
=\alpha=\frac{(r+1)^{p-r}}{\big(\lambda^{(p+1)}(G^{r+1})\big)^{p+1}},    
\]
as desired.

For the right inequality, we can prove $\lambda^{(p+1)}(H)\leq\lambda^{(p+1)}(G*K_1)$ similarly. 
According to \autoref{thm:G1*G2}, we have
\[
\lambda^{(p+1)}(G*K_1)=\left(\frac{(r+1)^{p-r}}{r^{p+1-r}}\right)^{1/(p+1)}\lambda^{(p+1)}(G), 
\]
the result follows. 
\end{proof}

\section{The $\alpha$-normal labeling method for $1\leq p<r$}

In this section, we make a brief discussion on the $\alpha$-normal labeling 
method for $1\leq p<r$. Due to the fact that the Perron--Frobenius Theorem 
fails for general hypergraph $G$ when $1\leq p<r$, the theory is less effective 
than the case $p\geq r$. However, we can still define the $\alpha$-normal 
labeling method as \autoref{def:consistent normal}.

Unlike the case $p> r$, neither the existence nor the uniqueness can be said for 
the $\alpha$-normal labeling for general $r$-uniform hypergraph $G$. However, we 
still have the following result.

\begin{theorem}\label{p<r}
For $1\leq p<r$, and any $r$-uniform hypergraph $G$ with $p$-spectral radius 
$\lambda^{(p)}(G)$, there exists an induced sub-hypergraph $G[S]$ such that 
$G[S]$ is consistently $\alpha$-normal with $\alpha=r^{p-r}/(\lambda^{(p)}(G))^p$.

Conversely, we have
\[
\lambda^{(p)}(G))=r^{1-r/p} \max_i\big\{\alpha_i^{-1/p}\big\},
\]
where the maximum is taken over all $\alpha_i$ such that there is a consistent 
$\alpha_i$-normal labeling on some induced sub-hypergraph of $G$.
\end{theorem}

\begin{proof}
Assume that $P_G(\bm{x})$ reaches the maximum at
$\bm{x}^*=(x_1,x_2,\ldots,x_n)^{\mathrm{T}}\in\mathbb{S}^{n-1}_{p,+}$.
Let $S=\{i\colon x_i>0\}$. Consider the induced hypergraph $G[S]$.
Observe that $P_{G[S]}(\bm{x})$ reaches the maximum at 
$\bm{x}^*|_S\in \mathbb{S}^{|S|-1}_{p,++}$, and therefore 
$\lambda^{(p)}(G[S])=\lambda^{(p)}(G)$.

Define a weighted incidence matrix $B$ and $\{w(e)\}$ (on $G[S]$) as follows:
\begin{align*}
B(v,e) & =\begin{dcases}
\frac{\prod_{u\in e}x_u}{\lambda^{(p)}(G[S]) x_v^p}, & \text{if}~v\in e
~\text{and}~ v\in S,\, e\in E(G[S]),\\
0, & \text{otherwise},
\end{dcases}\\[2mm]
w(e) & =\frac{r\prod_{u\in e}x_u}{\lambda^{(p)}(G[S])}.
\end{align*}
Since $x_v\not=0$ for any $v\in S$, the above $B$ and $\{w(e)\}$ are well-defined
on $G[S]$.

For any $v\in S$, using the eigenequation \eqref{eq:Eigenequation} gives
\[
\sum_{e:\,v\in e}B(v,e)=
\frac{\sum_{e:\,v\in e}\prod_{u\in e}x_u}{\lambda^{(p)}(G[S]) x_v^p}=1.
\]
Also, we see that
\[
\sum_{e\in E(G[S])}w(e)=\frac{r}{\lambda^{(p)}(G[S])}\sum_{e\in E(G[S])}\prod_{u\in e}x_u
=\frac{\lambda^{(p)}(G[S])}{\lambda^{(p)}(G[S])}=1.
\]
Therefore items (1) and (2) of \autoref{def:consistent normal} are verified. For 
item (3), we check that 
\begin{align*}
w(e)^{p-r}\cdot\prod_{v\in e}B(v,e) & =
\Bigg(\frac{r}{\lambda^{(p)}(G[S])}\prod_{u\in e}x_u\Bigg)^{p-r}\cdot
\prod_{v\in e}\frac{\prod_{u\in e}x_u}{\lambda^{(p)}(G[S])x_v^p}\\
& =\frac{r^{p-r}}{(\lambda^{(p)}(G[S]))^p}=\alpha.
\end{align*}
To show that $B$ is consistent, for any $v\in S$ and $v\in e_i$, $i=1,2,\ldots,d$, 
we have
\[
\frac{w(e_1)}{B(v,e_1)}=\frac{w(e_2)}{B(v,e_2)}=\cdots
=\frac{w(e_d)}{B(v,e_d)}=rx_v^p.
\]
    
Conversely, assume that for some $S_i\subset V$, $G[S_i]$ is consistently 
$\alpha_i$-normal with weighted incidence matrix $B$ and weights $\{w(e)\}$. 
Define a vector $\bm{x}=(x_1,x_2,\ldots,x_n)^{\mathrm{T}}\in\mathbb{S}_{p,+}^{n-1}$ 
for $G$ as follows:
\[
  x_v=
  \begin{dcases}
    \bigg(\frac{w(e)}{rB(v,e)}\bigg)^{1/p}, & \text{if}\ v\in e\in E(G[S_i]), \\
    0, & \text{otherwise}.
  \end{dcases}
\]
The consistent condition guarantees that $x_v$ (for $v\in S$) is independent 
of the choice of the edge $e$. Clearly, $||\bm{x}||_p=1$. Hence, 
\begin{align*}
	\lambda^{(p)}(G)\geq P_G(\bm{x}) & =r\sum_{e\in E(G[S_i])}\prod_{v\in e}x_v\\
	& =r^{1-r/p}\sum_{e\in E(G[S_i])}\frac{w(e)^{r/p}}{\prod_{v\in e}(B(v,e))^{1/p}}\\
	& =r^{1-r/p}\sum_{e\in E(G[S_i])}\frac{w(e)}{w(e)^{1-r/p}\prod_{v\in e}(B(v,e))^{1/p}}\\
	& =\frac{r^{1-r/p}}{\alpha_i^{1/p}}\sum_{e\in E(G[S_i])}w(e)\\
	& =\frac{r^{1-r/p}}{\alpha_i^{1/p}}.
\end{align*}
Combining with the first part of the theorem, we have
\[
\lambda^{(p)}(G))=r^{1-r/p} \max_i\{\alpha_i^{-1/p}\}.
\]
The proof is completed.
\end{proof}

\begin{example}
Consider the following graph $G$ with $6$ vertices and $7$ edges.
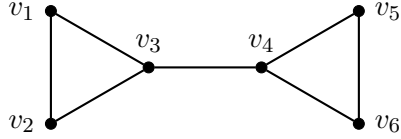
\begin{center}
  \begin{tikzpicture}[scale=1.5, vertex/.style={circle, draw=black, fill=white} ] 
    \node at (150:1) [vertex, scale=0.4] (v1) [fill=black, label=left:{$v_1$}] {};
    \node at (210:1) [vertex, scale=0.4] (v2) [fill=black, label=left:{$v_2$}] {};
    \node at (0,0) [vertex, scale=0.4] (v3) [fill=black, label=above:{$v_3$}] {};
    \node at (1,0) [vertex, scale=0.4] (v4) [fill=black,label=above:{$v_4$}] {};
    \node at (1.863,0.5) [vertex, scale=0.4] (v5) [fill=black,label=right:{$v_5$}] {};
    \node at (1.863,-0.5) [vertex, scale=0.4] (v6) [fill=black,label=right:{$v_6$}] {};
    \draw[thick] (v1)--(v2)--(v3)--(v4)--(v5)--(v6) (v1)--(v3) (v4)--(v6);
  \end{tikzpicture}
\end{center}

When $p=1$, $G[S]$ has a consistent $\alpha_i$-normal labeling if and only if $S$
forms a clique of size $2$ or $3$ in $G$. In particular, both $G[\{v_1,v_2,v_3\}]$ 
and $G[\{v_4,v_5,v_6\}]$ has a consistent $\frac{3}{4}$-normal labeling while 
$G[\{u,v\}]$ has a consistent $1$-normal labeling for each edge $uv$. We have
\[
\lambda^{(1)}(G)=2^{-1}\cdot \max\Big\{\frac{4}{3}, 1\Big\}=\frac{2}{3}.
\]
\end{example}

We can also define the $\alpha$-subnormal for $p\in [1,r)$ similar to the case for 
$p=r$ (see \cite[Definition 4]{LuMan2016:Small Spectral Radius}).

\begin{definition}\label{psubnormaldef}
For $1\leq p< r$, a hypergraph $G$ with $m$ edges is called {\em $\alpha$-subnormal} 
for $p$ if there exists a weighted incidence matrix $B$ satisfying
\begin{enumerate}
\item[$(1)$] $\displaystyle\sum_{e:\,v\in e}B(v,e)\leq 1$, for any $v\in V(G)$;
\item[$(2)$] $m^{r-p}\displaystyle\prod_{v\in e}B(v,e)\geq\alpha$, for any $e\in E(G)$.
\end{enumerate}
\end{definition}

\begin{theorem}\label{subt<r}
Let $G$ be an $r$-uniform hypergraph with $m$ edges. If $G$ is $\alpha$-subnormal for 
$p\in [1,r)$, then the $p$-spectral radius of $G$ satisfies
\[
\lambda^{(p)}(G)\leq \frac{(r/m)^{1-r/p}}{\alpha^{1/p}}.   
\]
\end{theorem}

\begin{proof}
For any nonnegative vector $\bm{x}=(x_1,x_2,\ldots,x_n)^{\mathrm{T}}\in\mathbb{S}^{n-1}_{p,+}$, 
we have
\begin{align*}
r\sum_{\{i_1,\ldots,i_r\}\in E(G)}x_{i_1}\cdots x_{i_r} & \leq m^{r/p-1} 
\frac{r}{\alpha^{1/p}}\sum_{e\in E(G)}\prod_{v\in e}\big(B(v,e)\big)^{1/p}x_v\\
& \leq m^{r/p-1}\frac{r}{\alpha^{1/p}}
\Bigg(\sum_{e\in E(G)}\prod_{v\in e}\big(B(v,e)\big)^{1/r}x_v^{p/r}\Bigg)^{r/p}\\
& \leq m^{r/p-1}\frac{r^{1-r/p}}{\alpha^{1/p}}\Bigg(\sum_{e\in E(G)}\sum_{v\in e}B(v,e)x_v^p\Bigg)^{r/p}\\
& \leq \frac{(r/m)^{1-r/p}}{\alpha^{1/p}},
\end{align*}
which yields $\lambda^{(p)}(G)\leq (r/m)^{1-r/p}\alpha^{-1/p}$. 
\end{proof}

From Definition \ref{psubnormaldef}, a hypergraph $G$ (with $m$ edges) is $\alpha$-subnormal 
for $p\in [1,r)$ if and only if $G$ is $\alpha'$-subnormal for any $p'\in [1,r)$ with
$\alpha'=\alpha m^{p-p'}$. We have the following corollary.

\begin{corollary}
Let $G$ be an $r$-uniform hypergraph with $m$ edges. If $G$ is $\alpha$-subnormal for 
$1\leq p<r$, then for any $p'\in [1,r)$ the $p'$-spectral radius of $G$ satisfies
\[
\lambda^{(p')}(G)\leq \frac{r^{1-r/p'}}{\alpha^{1/p'} m^{(p-r)/p'}}.   
\] 
\end{corollary}

For $1\leq p<r$, denote $\mathcal{G}^r(p)$ the set of $r$-uniform hypergraph $G$ for which
$P_G(\bm{x})$ reaches the maximum at some point in $\bm{x}\in\mathbb{S}_{p,++}^{|V(G)|-1}$.
Then for any $r$-uniform hypergraph $G$, there exists a set $S\subset V$ such that
$G[S]\in \mathcal{G}^r(p)$. Finally, we conclude this section with a problem of Nikiforov 
\cite[Problem 5.9]{Nikiforov2014:Analytic Methods}, which are related to the topic of this 
section.

\begin{problem}
Given $1\leq p<r$, characterize all $r$-uniform hypergraphs in $\mathcal{G}^r(p)$.  
\end{problem}


\begin{thebibliography}{99}
\bibitem{BaiLu2017}
S. Bai, L. Lu, A bound on the spectral radius of hypergraphs with $e$
edges, Linear Algebra Appl. (2018),  https://doi.org/10.1016/j.laa.2018.03.030.
\bibitem{Bretto2013}
A. Bretto, Hypergraph Theory: An Introduction, Springer, 2013.
    
\bibitem{ChangDing2017}
J. Chang, W. Ding, L. Qi, H. Yan, Computing the $p$-spectral radii of uniform hypergraphs with applications,  J. Sci. Comput. Volume 75, Issue 1, (2018), pp 1–25.
    
\bibitem{Cooper2012}
J. Cooper, A. Dutle, Spectra of uniform hypergraphs, Linear Algebra Appl. 436 (2012)
3268--3299.
 
\bibitem{Kang2014}
L. Kang, V. Nikiforov, Extremal problems for the $p$-spectral radius of graphs,
Electron. J. Combin. 21(3) (2014) P3.21. 
       
\bibitem{KangLiu2016}
L. Kang, L. Liu, L. Qi and X. Yuan, Some results on the spectral radii of uniform hypergraphs, 
preprint available at \href{https://arxiv.org/abs/1605.01750}{\texttt{arXiv:\,1605.01750}}, 2016.
    
\bibitem{Keevash2014}
P. Keevash, J. Lenz, D. Mubayi, Spectral extremal problems for hypergraphs, 
SIAM J. Discrete Math. 28 (4) (2014) 1838--1854.
     
\bibitem{LiShao2016}
H. Li, J. Shao, L. Qi, The extremal spectral radii of $k$-uniform supertrees, J. Comb. Optim. 32
(2016) 741--764.
         
\bibitem{LiuKang2016}
L. Liu, L. Kang, X. Yuan, On the principal eigenvectors of uniform
hypergraphs, Linear Algebra Appl. 511 (2016) 430--446.
        
\bibitem{LuMan2016:Small Spectral Radius} 
L. Lu, S. Man, Connected hypergraphs with small spectral radius, Linear Algebra Appl.
509 (2016) 206--227.
    
\bibitem{Nikiforov2014:Analytic Methods}
V. Nikiforov, Analytic methods for uniform hypergraphs, Linear Algebra Appl. 457
(2014) 455--535.

\bibitem{Nikiforov2014:Extremal Problems}
V. Nikiforov, Some extremal problems for hereditary properties of graphs, 
Electron. J. Combin. 21 (2014) P1.17.

\bibitem{OuyangQi2017}
C. Ouyang, L. Qi, X. Yuan, The first few unicyclic and bicyclic hypergraphs with largest spectral radii,
Linear Algebra Appl. 527 (2017) 141--162.

\bibitem{Shao2013}
J. Shao, A general product of tensors with applications,
Linear Algebra Appl. 439 (2013) 2350--2366.
   
\bibitem{Talbot2002}
J. Talbot, Lagrangians of hypergraphs, Combin. Probab. Comput. 11 (2002) 199--216.

\bibitem{XiaoWang2018}
P. Xiao, L. Wang, Y. Du, The first two largest spectral radii of uniform
supertrees with given diameter, Linear Algebra Appl. 536 (2018) 103--119.
     
\bibitem{YuanSi2017}
X. Yuan, X. Si, L. Zhang, Ordering uniform supertrees by their
spectral radii, Front. Math. China 12(6) (2017) 1393--1408.
    
\bibitem{ZhangKang2017}
W. Zhang, L. Kang, E. Shan et al., The spectra of uniform hypertrees,
Linear Algebra Appl. 533 (2017) 84--94.
    
\end{thebibliography}
\end{document}